\documentclass[12pt, reqno]{amsart} 
\usepackage{amsthm,amssymb}
\usepackage{graphicx}
\usepackage{enumerate}
\usepackage{amssymb,amsmath,graphicx,amsfonts,euscript}
\usepackage{color}

\usepackage[margin=3cm, a4paper]{geometry}

 \def\norm#1{\|#1\|}

\def\H{{\cal H}}

\def\H1{H^1(\R)}

\newcommand{\les}{\lesssim}

 \newcommand{\R}{\mathbb{R}}

\newcommand{\al}{\alpha} 
 
\newcommand{\e}{\varepsilon} 
 \newcommand{\la}{\lambda}

\newcommand{\p}{\partial} 

\newcommand{\im}{\mathop{\mathrm{Im}}}

 \newcommand{\Del}[1]{}

\numberwithin{equation}{section}

\newtheorem{thm}{Theorem}[section] 

\newtheorem{lem}[thm]{Lemma}

\theoremstyle{remark} 

\newtheorem*{exam*}{Examples}

\begin{document}

\setcounter{page}{1}

\title[GWP for DNLS]{Global well-posedness for the derivative nonlinear 
Schr\"{o}dinger equation in $H^{\frac 12} (\R)$}

\author[Z. Guo]{Zihua Guo}
\address{School of Mathematical Sciences, Monash University, VIC 3800, Australia \& LMAM, School of Mathematical Sciences, Peking
University, Beijing 100871, China}
\email{zihua.guo@monash.edu}

\author{Yifei Wu}
\address{Center for Applied Mathematics, Tianjin University, Tianjin 300072, China}
\email{yerfmath@gmail.com}
\subjclass[2010]{Primary  35Q55; Secondary 35A01}


\keywords{Nonlinear Schr\"{o}dinger equation with derivative,
global well-posedness, low regularity}

\maketitle

\begin{abstract}\noindent
We prove that the derivative nonlinear 
Schr\"{o}dinger equation is globally well-posed in $H^{\frac 12} (\R)$ when the mass of initial data 
is strictly less than $4\pi$. 
\end{abstract}

\section{Introduction}
In this note, we study the Cauchy problem to the derivative nonlinear Schr\"odinger equation (DNLS):
\begin{align}\label{eq:DNLS}
\begin{split}
i\partial_{t}u+\partial_{x}^{2}u=&i\partial_x(|u|^2u),\qquad t\in \R,\ x\in \R,\\
u(0,x) =&u_0(x).
\end{split}
\end{align}
This equation was derived by \cite{MioOgino76, Mjolhus76} for studying the propagation of the circular polarised nonlinear Alfv\'en waves in plasma, and has been extensively studied since then.  It is well-known that \eqref{eq:DNLS} is completely integrable (see \cite{KaupNewell78,ko,tsuchida}), and thus has infinite number of conservation laws.  In particular, in this paper we will use the following three conservation laws: if $u$ is a $H^1$-solution of \eqref{eq:DNLS} then
\begin{align*}
M_D(u):=&\int_{\R}|u|^2dx=M_D(u_0),\\
E_D(u):=&\int_{\R}|u_x|^2+\frac{3}{2}|u|^2\im (u\bar u_x)+\frac{1}{2}|u|^6dx=E_D(u_0),\\
P_D(u):=&\int_{\R}\im(\bar uu_x)-\frac{1}{2}|u|^4dx=P_D(u_0).
\end{align*}

Equation \eqref{eq:DNLS} has been extensively studied. On the well-posedness,
Hayashi and Ozawa \cite{Ha-93-DNLS, HaOz-92-DNLS, HaOz-94-DNLS, Oz-96-DNLS} proved local well-posedness in $H^1(\R)$, and moreover global well-posedness for initial data in $H^1$ satisfying
\begin{align}\label{eq:massth1}
\int_{\R} |u_0(x)|^2dx<2\pi.
\end{align}
The condition above appears naturally in the sharp Galiardo-Nirenberg inequality to ensure an apriori estimate of $H^1$-norm by mass and energy conservation. Later, Local well-posedness in $H^s$ for $s\geq 1/2$ was obtained by Takaoka \cite{Ta-99-DNLS-LWP}, and this result is sharp in the sense that the solution map fails to be uniformly continuous in a ball of $H^s$ if $s<1/2$. Low regularity global well-posedness was also studied, for example, global well-posedness in $H^s(\R)$ under \eqref{eq:massth1} was obtained in \cite{Ta-01-DNLS-GWP, CKSTT-01-DNLS, CKSTT-02-DNLS} for $s>1/2$, and finally in \cite{Miao-Wu-Xu:2011:DNLS} for $s=1/2$. On the long-time behavior and modified scattering theory, see \cite{GHLN} and references therein.

A natural question is whether blowup occurs for \eqref{eq:DNLS}. To the authors' knowledge, this problem is still open. See \cite{LSS} for a numerical blowup analysis on a class of DNLS. Recently, the second author \cite{Wu2} showed the global well-posedness in $H^1$ under a weaker condition
\begin{align}\label{eq:massth2}
\int_{\R} |u_0(x)|^2dx<4\pi,
\end{align}
improving his previous result \cite{Wu1}. This result shows a striking difference between DNLS and other mass critical equations like focusing generalized KdV and quintic focusing nonlinear Schr\"odinger equation. The key ingredient is the use of the momemtum conservation.  
   
The purpose of this paper is to prove the low-regularity global well-posedness under \eqref{eq:massth2}.
The main result is

\begin{thm}\label{thm:main}
The Cauchy problem \eqref{eq:DNLS} is global well-posed in $H^\frac12(\R)$ under \eqref{eq:massth2}.
\end{thm}

We explain the ideas of the proof of the theorem.  Inspired by \cite{Wu2}, we derive directly an apriori estimate using the conservation laws of mass, momentum and energy as well as the sharp Gagliardo-Nirenberg inequality, and thus provide a simplified proof of the result of \cite{Wu2}. We do not prove by contradiction and can get a clear bound of $H^1$-norm. Then we combine it with the I-method to prove the theorem.

\section{Apriori estimate}

To prove the theorem, it suffices to control the $H^{1/2}$-norm of the solution. For convenience, we use the following gauge transformation. If $u$ is a solution to \eqref{eq:DNLS} with $u_0\in H^{1/2}$, let
\begin{align}\label{eq:gauge}
v(t,x):=e^{-\frac{3}{4}i\int_{-\infty}^x |u(t,y)|^2\,dy}u(t,x).
\end{align}
Then $v$ solves
\begin{align}\label{eq:DNLS2}
i\partial_{t}v+\partial_{x}^{2}v=\frac i2|v|^2v_x-\frac i2v^2\bar{v}_x-\frac{3}{16}|v|^4v
\end{align}
with initial data $v(0,x)=v_0(x):=e^{-\frac{3}{4}i\int_{-\infty}^x |u_0|^2\,dy}u_0$.  It's easy to see the map $u\to v$ is a bijection in $H^{1/2}$. Indeed, by fractional Leibniz rule we get
\begin{align*}
\norm{D^{1/2} v}_{2}\les& \norm{D^{1/2} u}_2+\norm{uD^{1/2} [e^{-\frac{3}{4}i\int_{-\infty}^x |u(t,y)|^2\,dy}]}_2\\
\les& \norm{D^{1/2} u}_2+\norm{u}_{4} \norm{\partial_x [e^{-\frac{3}{4}i\int_{-\infty}^x |u(t,y)|^2\,dy}]}_{4/3}\les C(\norm{u}_{H^{1/2}}).
\end{align*}
From now on, we only consider the equation \eqref{eq:DNLS2} and we need to control the $H^{1/2}$-norm of $v$. 

Under the gauge transformation, the conservation laws reduce to: for solution $v$ of \eqref{eq:DNLS2} then
\begin{align}
M(v(t)):=&\|v(t)\|_{L^2_x}^2=
M(v_0),\qquad \mbox{(mass)}\\
P(v(t)):=&\mbox{Im} \int_{\R}\bar v(t) v_x(t)\,dx+\frac 14\int_{\R}
|v(t)|^4\,dx=P(v_0),\qquad \mbox{(momentum)}\\
E(v(t)):=&\|v_x(t)\|_{L^2_x}^2-\frac{1}{16}\|v(t)\|^6_{L^6_x}=
E(v_0).\qquad \mbox{(energy)}
\end{align}
We denote $\norm{\cdot}_p=\norm{\cdot}_{L_x^p}$ for $1\leq p\leq \infty$. By the sharp Galiardo-Nirenberg inequality
\begin{align}\label{eq:GN2}
\norm{f}_6^6\leq& \frac{4}{\pi^2}\norm{f}_2^{4}\norm{f_x}_2^{2},
\end{align}
then we get
\[E(v)\geq \|v_x\|_{2}^2(1-\frac{1}{4\pi^2}\norm{v}_2^4).\]
Thus under the condition \eqref{eq:massth1} we can get the apriori bound on $\norm{v}_{H^1}$.

However, as observed in \cite{Wu2} the momentum conservation for \eqref{eq:DNLS2} played a significant role. Inspired by \cite{Wu2} we derive directly a-priori estimate using the momentum and the following sharp GN inequality (see \cite{Agueh}):
\begin{align}\label{eq:GN4}
\norm{f}_6\leq& C_{GN}\norm{f}_4^{8/9}\norm{f_x}_2^{1/9},
\end{align}
where $C_{GN}=3^{\frac{1}{6}}(2\pi)^{-\frac{1}{9}}$.

\begin{lem}\label{lem:PuL4}
If $v\in H^1(\R)$ and $v\ne 0$, then
\begin{align}\label{eq:PuL4}
P(v)\geq& \frac{1}{4}\norm{v}_4^4(1-\frac{1}{2\sqrt{\pi}}\norm{v}_2)-\frac{4\sqrt{\pi}E(v)\norm{v}_2}{\norm{v}_4^4}.
\end{align}
\end{lem}
\begin{proof}
Let $u=e^{i\al x}v(t,x)$ with $\al>0$ being determined later. Then
\[|u_x|^2=|v_x|^2+\al^2|v|^2+2\al \im v_x\bar v,\]
and thus
\[\int \im v_x\bar vdx=-\frac{E(v)}{2\al}-\frac{\al M(v)}{2}+\frac{E(u)}{2\al}.\]
Now by the sharp GN inequality we have
\begin{align*}
E(u)=&\norm{u_x}_2^2-\frac{1}{16}\norm{u}_6^6\\
\geq& C_{GN}^{-18}\norm{u}_6^{18}\norm{u}_4^{-16}-\frac{1}{16}\norm{u}_6^6\\
=&(C_{GN}^{-18} \norm{v}_6^{12}\norm{v}_4^{-16}-\frac{1}{16})\norm{v}_6^6.
\end{align*}
Thus,
\begin{align*}
P(v)\geq& -\bigg[\frac{1}{16}-C_{GN}^{-18} \norm{v}_6^{12}\norm{v}_4^{-16}\bigg]\frac{\norm{v}_6^6}{2\al}+\frac{\norm{v}_4^4}{4}-\frac{\al \norm{v}_2^2}{2}-\frac{E(v)}{2\al}\\
\geq&-f(\norm{v}_6^6\norm{v}_4^{-8})\frac{\norm{v}_4^{8}}{2\al}+\frac{\norm{v}_4^4}{4}-\frac{\al \norm{v}_2^2}{2}-\frac{E(v)}{2\al}
\end{align*}
where $f(x)=(\frac{1}{16}-C_{GN}^{-18}x^2)x$.  By calculus we know
\[\max_x f(x)=f(\frac{C_{GN}^9}{4\sqrt{3}})=\frac{C_{GN}^9}{96\sqrt{3}}=\frac{1}{64\pi}.\]
Therefore
\begin{align*}
P(v)\geq& -\frac{\norm{v}_4^{8}}{128\pi\al}+\frac{\norm{v}_4^4}{4}-\frac{\al \norm{v}_2^2}{2}-\frac{E(v)}{2\al}
\end{align*}
Take $\al=\frac{1}{8\sqrt{\pi}}\norm{v}_4^4\norm{v}_2^{-1}$, then
$P(v)\geq \frac{1}{4}\norm{v}_4^4(1-\frac{1}{2\sqrt{\pi}}\norm{v}_2)-\frac{E(v)}{2\al}$.
\end{proof}

\begin{lem}
If $v\in H^1(\R)$, $v\ne 0$ and $\norm{v}_2^2<4\pi$, then
\begin{align}\label{eq:H1bd}
\|v_x\|_{L^2}^2\le 2 E(v)+\frac{P(v)^2+2\sqrt{\pi}|E(v)|\norm{v}_2}{(1-\frac{1}{2\sqrt{\pi}}\norm{v}_2)^{2}}.
\end{align}
\end{lem}
\begin{proof}
Let $x=\norm{v}_4^4$. Then \eqref{eq:PuL4} gives a estimate of the form
\begin{align}\label{eq:PuL4v2}
c\geq ax-\frac{b}{x}.
\end{align}
with $a=\frac{1}{4}(1-\frac{1}{2\sqrt{\pi}}\norm{v}_2)$, $b=4\sqrt{\pi}|E(v)|\norm{v}_2$, $c=|P(v)|$. 
\eqref{eq:PuL4v2} implies \[(ax^2-cx-b)\leq 0.\]
Since $a>0$, thus we get
\[x^2\leq \Big(\frac{c+\sqrt{c^2+4ab}}{2a}\Big)^2\leq \frac{c^2+2ab}{a^2}.\]
Thus we obtain
\begin{align}\label{eq:L4bd}
\norm{v}_4^8\leq 16 (1-\frac{1}{2\sqrt{\pi}}\norm{v}_2)^{-2}\Big(P(v)^2+2(1-\frac{1}{2\sqrt{\pi}}\norm{v}_2)\sqrt{\pi}|E(v)|\norm{v}_2\Big).
\end{align}
Then by \eqref{eq:GN4} and mean value inequality we have
\begin{align}
\|v_x\|_{L^2}^2\le 2 E(v)+2^{-4}\|v\|_{L^4}^8.
\end{align}
Therefore by \eqref{eq:L4bd} we prove the lemma.
\end{proof}

With this lemma, we can get that if $v$ is a $H^1$-solution of \eqref{eq:DNLS2} satisfying \eqref{eq:massth2}, then $\norm{v_x}_2\leq C$. 
Therefore, global well-posedness of \eqref{eq:DNLS2} in $H^1$ under \eqref{eq:massth2} follows immediately.

\section{Proof of the main theorem}

In this section we prove Theorem \ref{thm:main} using the I-method as the previous works \cite{CKSTT-02-DNLS, Miao-Wu-Xu:2011:DNLS}.  The main difference is that we need to use the momentum conservation.

First we recall the definition of $I$-operator.
Let $N\gg 1$ be fixed, and the Fourier multiplier operator
$I_{N}$ be defined as
\begin{equation}
\widehat{I_Nf}(\xi)=m_N(\xi)\hat{f}(\xi).\label{I}
\end{equation}
Here $m_N(\xi)$ is a smooth, radially decreasing function
satisfying $0<m_N(\xi)\leq 1$ and
\begin{equation}
m_N(\xi)=\biggl\{
\begin{array}{ll}
1,&| \xi|\leq N,\\
N^\frac12| \xi|^{-\frac12},&| \xi|>2N.\label{m}
\end{array}
\end{equation}
For simplicity we denote $I_N$ by $I$ and $m_N$ by $m$
if there is no confusion. $I_N$ maps $H^\frac12$ to $H^1$, moreover, we have the following estimates,
\begin{equation}
    \|f\|_{H^\frac12}\lesssim \|I_Nf\|_{H^1}
\lesssim
    N^\frac12\|f\|_{H^{\frac{1}{2}}},\label{II}
\end{equation}
where the implicit constants are indenpendent on $N$. 

Next we use the rescaling. For $v_0\in H^{1/2}$, let $v$ be the solution to \eqref{eq:DNLS2}. For $\lambda>0$, let 
$$
v_\lambda=\lambda^{-\frac12} v(\frac{x}{\lambda}, \frac{t}{\lambda^2})\quad \mbox{ and } \quad
v_{0,\lambda}=\lambda^{-\frac12} v_0(\frac{x}{\lambda}).
$$
Then $v_\lambda$ is a solution of \eqref{eq:DNLS2} with the
initial data $v_\lambda(0)= v_{0,\lambda}(x)$. Meanwhile, $v_\lambda$ exists on
$[0,T]$  if and only if  $v$ exists
on $[0,\lambda^{-2}T]$. We have
\begin{align}\label{eq:L2norm}
\|Iv_{0,\lambda}\|_{2}\leq \|v_{0,\lambda}\|_{2}=\norm{v_0}_2
\end{align}
and
\begin{align}\label{eq:H1norm}
\|\p_xIv_{0,\lambda}\|_{2}\les N^{1/2}\la^{-1/2}\norm{v_0}_{\dot H^{1/2}}.
\end{align}
Thus choosing 
$$
\lambda\sim N,
$$
we can make
\begin{align}\label{Rescaling}
\|\partial_x Iv_{0,\lambda}\|_{2}\leq \e_0\ll 1
\end{align}
where $\e_0$ will be determined later. 

We recall a variant local well-posedness obtained in \cite{Miao-Wu-Xu:2011:DNLS}. 
\begin{lem} \label{prop:modified-local}
The Cauchy problem \eqref{eq:DNLS2} is
locally well-posed for the initial data $v_0$ satisfying $Iv_0\in
H^1(\R)$. Moreover, the solution exists on the interval $[0,\delta]$
with the lifetime
\begin{equation}
 \delta\sim \|Iv_0\|^{-\mu}_{H^1}\label{delta}
\end{equation}
for some $\mu>0$, where the implicit constant is independent of $N$. Furthermore, the
solution satisfies the estimate
\begin{equation}
\|Iv\|_{L^\infty((0,\delta);H^1)}\le 2\|Iv_0\|_{H^1}.\label{LSE}
\end{equation}
\end{lem}

By the above lemma, we need to control the growth of $\|Iv_\lambda(t)\|_{H^1}$. By mass conservation we have $\|Iv_\lambda(t)\|_{L_x^2}\leq \|v_\lambda\|_{L_x^2}\leq C$. It suffices to control $\|\partial_xIv_\lambda\|_{2}$. We will use \eqref{eq:H1bd} since $\|Iv_\lambda\|_{2}^2\leq \|v_\lambda\|_{2}^2=\norm{v_0}_2^2<4\pi$.
We define the modified momentum and energy as follows
\begin{align}
P_I(v_\lambda):=P(Iv_\lambda),\quad E_I(v_\lambda):=E(Iv_\lambda).
\end{align}
Then by \eqref{Rescaling}, H\"older's and Sobolev's inequalities,  we have 
$$
P_I(v_{0,\lambda})\lesssim 1;\quad 
E_I(v_{0,\lambda})\lesssim 1
$$
Moreover, 
$$
P(v_{0,\lambda})=\frac1\lambda P(v_0)\sim N^{-1}P(v_0).
$$
If $N\to \infty$, $I_N$ tends to the identity operator. Thus $P_I(v_\lambda)$ and $E_I(v_\lambda)$ increases slowly in $t$ if $N$ is large enough. Indeed, in the previous works the growth of $E_I(v_\lambda)$ was already studied. Collecting the results obtained in \cite{Miao-Wu-Xu:2011:DNLS} (see Section 7), we have 
\begin{lem}\label{lem:ME}
Suppose that for $T>0$
\begin{align}\label{assumption}
\sup\limits_{t\in [0,T]} \|Iv_\lambda\|_{H^1}\lesssim  1,
\end{align}
then the modified energy $E_I(v_\lambda)$ obeys the following estimate: there exists $C,\alpha>0$ such that for any $t\in [0,T]$ and any $\e>0$ 
\begin{align}\label{Est:ME}
|E_I(v_\lambda(t))|\le &\|\partial_xIv_{0,\lambda}\|_{L^2}^2+CN^{-\alpha}\sup\limits_{\tau\in [0,t]}\big(\|Iv_\lambda(\tau)\|_{H^1}^4+\|Iv_\lambda(\tau)\|_{H^1}^6\big)\notag\\
&\quad + Ct N^{-\frac52+\epsilon}\sup\limits_{\tau\in [0,t]}\big(\|Iv_\lambda(\tau)\|_{H^1}^6+\|Iv_\lambda(\tau)\|_{H^1}^{10}\big).
\end{align}
\end{lem}

On the modified momentum we have the following estimate. Indeed, since the momentum lies in the regularity of $H^{1/2}$, we can estimate it in a simple way.
\begin{lem}\label{lem:MM}
We have
$$
\big|P_I(v_\lambda)-P(v_{\lambda})\big|\lesssim N^{-1}\big(\|Iv_\lambda\|_{H^1}^2+\|Iv_\lambda\|_{H^1}^4\big)
$$
\end{lem}
\begin{proof}
By the definition of momentum, we need to bound
\[\Big|\im\int_{\R}(I\bar v_\lambda \partial_xIv_\lambda-\bar v_\lambda \partial_xv_\lambda)\,dx\Big|+\Big|\int |Iv_\lambda|^4\,dx-\int |v_\lambda|^4\,dx\Big|:=I+II.\]
For the first term $I$, since 
\begin{align*}
\im\int_{\R}\Big(I\bar v_\lambda \partial_xIv_\lambda-\bar v_\lambda \partial_xv_\lambda\Big)\,dx
=\im \int_{\R}\Big(I\partial_x v_\lambda - \partial_x v_\lambda\Big)\big(I\bar v_\lambda+\bar v_\lambda\big)\,dx,
\end{align*}
and $P_{\le N}\big(I\partial_x v_\lambda - \partial_x v_\lambda\big)=0$, then we get 
\begin{align*}
I\lesssim &
\big\|I\partial_x v_\lambda - \partial_x v_\lambda\big\|_{\dot H^{-\frac12}}\Big(\big\|P_{\ge N}I\bar v_\lambda\big\|_{\dot H^{\frac12}}+\big\|P_{\ge N}\bar v_\lambda\big\|_{\dot H^\frac12}\Big).
\end{align*}
By the definition of $I$-operator, we have
\begin{align*}
\big\|I\partial_xv_\lambda-\partial_xv_\lambda\big\|_{\dot H^{-\frac12}}+\big\|P_{\ge N}I\bar v_\lambda\big\|_{\dot H^\frac12}+\big\|P_{\ge N}\bar v_\lambda\big\|_{\dot H^\frac12}&\lesssim N^{-\frac12}\|Iv_\lambda\|_{H^1},
\end{align*}
and thus
$$
I\lesssim N^{-1} \|Iv_\lambda\|_{H^1}^2.
$$

For the second term $II$, we have
\begin{align}\label{eq:PuII}
\int |Iv_\lambda|^4\,dx-\int |v_\lambda|^4\,dx=\int \big(Iv_\lambda-v_\lambda\big)\> P_{\ge N }( v_\lambda^3)\,dx+\mbox{similar terms}.
\end{align}
Using the H\"older inequality, the Sobolev's embedding, and the fractional Leibniz inequalities, we get
\begin{align*}
\Big|\int \big(Iv_\lambda-v_\lambda\big)\> P_{\ge N }( v_\lambda^3)\,dx\Big|
\lesssim & \|Iv_\lambda-v_\lambda\|_{6}\big\|P_{\ge N }( v_\lambda^3)\big\|_{\frac{6}{5}}\\
\lesssim & \|Iv_\lambda-v_\lambda\|_{H^\frac12}\>N^{-\frac12}\big\|\langle\nabla\rangle^{\frac12}( v_\lambda^3)\big\|_{\frac{6}{5}}\\
\lesssim & N^{-1}\|Iv_\lambda\|_{H^1}\>\big\|\langle\nabla\rangle^{\frac12}v_\lambda\big\|_{2}\|v_\lambda\|_{6}^2\\
\lesssim & N^{-1}\|Iv_\lambda\|_{H^1}\|v_\lambda\|_{H^\frac12}^3\\
\lesssim & N^{-1}\|Iv_\lambda\|_{H^1}^4.
\end{align*}
The similar terms in \eqref{eq:PuII} can be handled in the same way. Thus we prove the lemma.
\end{proof}

By Lemma \ref{lem:ME} and the mass conservation law $\norm{v_\lambda}_2\leq C$ we have under the assumption \eqref{assumption}
\begin{align}\label{Est:MEbd}
E_I(v_\lambda(t))\le &\|\partial_xIv_{0,\lambda}\|_{L^2}^2+CN^{-\alpha}\sup\limits_{\tau\in [0,t]}\big(\|\p_xIv_\lambda(\tau)\|_{2}^4+\|\p_xIv_\lambda(\tau)\|_{2}^6+1\big)\notag\\
&\quad +Ct N^{-\frac52+\epsilon}\sup\limits_{\tau\in [0,t]}\big(\|\p_xIv_\lambda(\tau)\|_{2}^6+\|\p_xIv_\lambda(\tau)\|_{2}^{10}+1\big).
\end{align}
Note that \eqref{Rescaling}. We will prove by continuity argument that for $T\le T_0:=N^{\frac52-2\epsilon}$, 
\begin{align}\label{eq:ASM}
\sup\limits_{t\in [0,T]}\|\partial_xIv_\lambda(t)\|_{2}\le 4\gamma_0\e_0,
\end{align}
where $\gamma_0=\sqrt{1+\frac{\sqrt{\pi}\norm{v_0}_2}{(1-\frac{1}{2\sqrt{\pi}}\norm{v_0}_2)^{2}}}$. We choose $\e_0\ll 1$ such that $100 \gamma_0\e_0<1$.

Assuming \eqref{eq:ASM}, we get that the
solution $v_\lambda$ exists on $[0,T_0]$.
Hence, $v$ exists on $[0,\lambda^{-2}T_0]$. Note that 
$$
\lambda^{-2}T_0\sim N^{-2}N^{\frac52-2\epsilon}= N^{\frac12-2\epsilon}.
$$
Therefore,we get that 
$v$ exists till arbitrarily large $T$ by choosing sufficient large $N$, and thus
completes the proof of Theorem \ref{thm:main}.

It remains to prove \eqref{eq:ASM}. We may assume $\sup\limits_{t\in [0,T]}\|\partial_xIv_\lambda(t)\|_{2}\le 3\e_0\ll 1$. Then the estimate \eqref{Est:MEbd} gives
\begin{align}
|E_I(v_\lambda(t))|\le \e_0^2+CN^{-\epsilon}, \quad t\leq T.
\end{align}
On the other hand, by Lemma \ref{lem:MM} we have
\begin{align}
|P_I(v_\lambda(t))|^2\le& 2|P_I(v_\lambda(t))-P(v_\lambda(t))|^2+2|P(v_\lambda(t))|^2
\le CN^{-2}.
\end{align}
By \eqref{eq:H1bd}, we have 
\begin{align*}
\big\|\partial_x Iv_\lambda(t)\big\|_{L^2}^2
\le& 2 E(Iv_\lambda(t))+\frac{P(Iv_\lambda(t))^2+2\sqrt{\pi}|E(Iv_\lambda)|\norm{v_0}_2}{(1-\frac{1}{2\sqrt{\pi}}\norm{v_0}_2)^{2}}\\
\le& 2 \gamma_0^2 (\e_0^2+CN^{-\epsilon})+CN^{-2}(1-\frac{1}{2\sqrt{\pi}}\norm{v_0}_2)^{-2}.
\end{align*}
Choosing $N$ sufficiently large, we obtain \eqref{eq:ASM} as desired.

\end{document}